\documentclass[a4paper,11pt,oneside,reqno]{amsart}
\usepackage[margin=1in]{geometry}                
\usepackage{lmodern} 
\usepackage[OT1]{fontenc}
\usepackage{graphicx}
\usepackage{amssymb,amsmath,amsthm}
\usepackage[numbers, square]{natbib}
\usepackage{enumerate}
\usepackage{hyperref}
\usepackage{color}
\usepackage{bbm}

\usepackage{epstopdf}
\newtheorem{theorem}{Theorem}
\newtheorem{remark}[theorem]{Remark}
\newtheorem{proposition}[theorem]{Proposition}
\newtheorem{corollary}[theorem]{Corollary}
\newtheorem{lemma}[theorem]{Lemma}

\DeclareMathOperator{\var}{\mathrm{var}}
\DeclareMathOperator{\rE}{\mathbb{E}}
\DeclareMathOperator{\rP}{\mathbb{P}}
\newcommand{\II}{\mathbf{I}}
\newcommand{\re}{\mathrm{e}}
\newcommand{\rd}{\mathrm{d}}
\newcommand{\ri}{\mathrm{i}}
\global\long\def\BB{\mathcal{B}}

\global\long\def\ZZ{\mathbb{Z}^{2}}

\global\long\def\CC{\mathcal{C}}

\global\long\def\NN{\mathbb{N}}

\global\long\def\cS{\mathsf{S}}

\global\long\def\PP{\mathbb{P}}

\begin{document}
\title[Relative Complexity of RWRS without local times]{Relative Complexity of Random Walks in Random Scenery in the absence of a weak invariance principle for the local times}
\author{George Deligiannidis}
\address{Department of Statistics, University of Oxford, Oxford OX1 3TG, UK}
\email{deligian@stats.ox.ac.uk}
\author{Zemer Kosloff}\thanks{The research of Z.K. was supported in part by the European Advanced
Grant StochExtHomog (ERC AdG 320977).}
\address{Mathematics Institute, University of Warwick, Coventry CV4 7AL, UK}
\email{z.kosloff@warwick.ac.uk}
\date{}                                           

\maketitle

\begin{abstract}
We answer the question of Aaronson about the relative complexity of Random Walks in Random Sceneries driven by either aperiodic two dimensional random walks, two-dimensional Simple Random walk, or by aperiodic random walks in the domain of attraction of the Cauchy distribution. A key step is proving that the range of the random walk satisfies the F\"olner property almost surely.
\end{abstract}
\section{Introduction}
The notion of entropy was first introduced into ergodic theory by Kolmogorov as an isomorphism invariant. That is if two measure preserving systems are (measure theoretically) isomorphic then their entropy is the same. It was later shown by a seminal theorem of Ornstein \cite{Ornstein} that entropy is a complete invariant for Bernoulli automorphisms (transformations which are isomorphic to a shift on an i.i.d. sequence, aka Bernoulli Shifts) meaning that two Bernoulli automorphisms are isomorphic if and only if their entropies coincide. Measure theoretic complexity, which is roughly the rate of growth of information, was introduced by Ferenczi~\cite{Ferenczi}, Katok and Thouvenot~\cite{Katok Thouvenot} and others as an isomorphism invariant for the problem of understanding whether two zero entropy systems are isomorphic. 

In a recent work, Aaronson \cite{Aaronson} introduced a relativised notion of complexity and calculated the relative complexity of random walks in random sceneries where the jump random variable is $\mathbb{Z}$ valued, centred, aperiodic and in the domain of attraction of an $\alpha$ stable distribution with $1<\alpha\leq 2$. The main tool used there is Borodin's weak invariance principle for the local times \cite{Borodin,Borodin 2}. Random Walks in Random Scenery(RWRS) are natural models for the study of this relative complexity notion as they are examples of non-Bernoulli $K$-automorphisms \cite{Kalikow} and so the relative complexity could be a good way to try to distinguish whether two such systems are isomorphic. Indeed, Aaronson's ideas of using the weak convergence of local times to count Hamming balls were later used by Austin \cite{Austin}, in the definition of a full isomorphism invariant, called the scenery entropy, for the class of random walk in random sceneries with jump distribution of finite variance. 

For the purpose of the introduction, we will now describe the classical random walk in random scenery from probability theory. Let $X_1,X_2,\dots$ be i.i.d. $\mathbb{Z}^d$-valued random variables,  the jump process, and $S_n:=\sum_{k=1}^{n}X_k$ the corresponding random walk. The scenery is an independent (of $\left\{ X_i \right\}$) field of i.i.d random variables $\left\{ C_j\right\}_{j\in\mathbb{Z}^d}$. The joint process $\left(X_n,C_{S_n}\right)$ is then known as a random walk in random scenery. The relative complexity of Aaronson in that case is heuristically as follows: Assuming that we have full information of the sequence $\mathbf{X}=X_1,X_2,...$, what is the rate of growth of the information arising from the scenery for most of the realizations of $X$?

If $X_1,X_2,...$ are Bernoulli $\pm1$ fair coin tossing (in other words the driving random walk is the simple random walk on the integers) then by the local central limit theorem, at time $n$, the range of the random walk $R(n):=\left\{ S_j:1\leq j\leq n\right\}$ is of order constant times $\sqrt{n}$. The range of the random walk is related to this problem since $\big\{C_{S_j}:j\in [1,n]\big\}=\big\{C_{k}:k\in R(n)\big\}$. It then appears that the rate of growth of information arising from the scenery should be of the order $\exp(H(C)\cdot \#R(n))$ with $H(C)$ the Shannon entropy of $C$. Thus for this example one would expect (which is verified in \cite{Aaronson}) that the relative complexity is of the order $\exp(c_w\sqrt{n}H(C))$, where $c_w$ is a constant depending on $w$ and the order should be interpreted as existence of a non trivial distributional limit. 

In this paper we treat random walks in random sceneries driven by aperiodic, recurrent,  random walks with finite second moments in $\mathbb{Z}^2$, by Simple Random Walk in $\mathbb{Z}^2$, or by an  aperiodic, recurrent, random walk in $\mathbb{Z}$ in the domain of attraction of the Cauchy distribution. Since the limiting distributions don't have local times, Aaronson's and Austin's methods do not apply. For these types of RWRS's, Kesten and Spitzer \cite{Kesten Spitzer} conjectured that there exists constants $a_n\to\infty$ such that
$\frac{1}{a_n}\sum_{k=1}^{nt}Z_{S_n}$ converges weakly to a Brownian motion (when $\var(C)<\infty$). This was shown to be true by Bolthausen \cite{Bolthausen} when $S_n$ is the two dimensional simple random walk on $\mathbb{Z}^2$ and by the first author and Utev \cite{DU15} for the case of the Cauchy distribution. Bolthausen's argument was generalized by \v{C}erny in \cite{cerny07} and the ideas there were a major inspiration for us. The idea is since one cannot have a weak invariance principle for the local time, one can study the asymptotics of self-intersection local times, see Section \ref{sec: range}, in order to prove a statement of the form "for most of the points of $R(n)$ the number of times the random walk visits them, up to time $n$, is greater than a constant times $\log(n)$" (see Theorem \ref{thm:distribution} for a precise statement). We refine this method to prove a result of independent interest, namely that the range of the random walk is almost surely a {F}\"olner sequence (Theorem \ref{thm:folner}). With these two Theorems at hand we can proceed by a simplified argument to deduce the main result, Theorem \ref{thm: relative complexity of RWRS}. which answers Aaronson's question about the relative complexity of this type of RWRS's. We think that this simpler (and softer) method can be used to calculate the relative complexity of other RWRS's such as \cite{Rudolph,Ball}.

This paper is organized as follows. In Section \ref{sec: Prelimaneries} we first start with the relevant definitions we need and then end with the statement of the main result. In Section \ref{sec: range} we prove the results we need for the random walk and it's range. Finally, Section \ref{sec: proof of main theorem} is the proof of the main Theorem. For the sake of completeness we include an Appendix with a proof of some standard facts about the random walks we consider.

\section{Preliminaries}\label{sec: Prelimaneries}

\subsection{Relative complexity over a factor}
Let $\left(X,\BB,m\right)$ be a standard probability space and $T:X\to X$
a $m$- preserving transformation. Denote by $\mathfrak{B}(X)$ the
collection of all measurable countable partitions of $X$. In order
to avoid confusion with notions from probability, we will denote the
partitions by Greek letters $\beta\in\mathfrak{B}(X)$ and the atoms
of $\beta$ by $\beta^{1},\beta^{2},...,\beta^{\#\beta},\ \#\beta\in\mathbb{N\cup\left\{ \infty\right\} }$.
A partition $\beta$ is a \textit{generating partition} if the smallest $\sigma$-algebra
containing $\left\{ T^{-n}\beta:\ n\in\mathbb{Z}\right\} $ is $\BB$.  For $\beta\in\mathfrak{B}(X)$ and $n\in\mathbb{N}$ let 
$$\beta_0 ^n:=\vee_{j=0}^n T^{-j}\beta=\left\{ \cap_{j=0} ^{n}T^{-j} b_j: b_1,..,b_n\in\beta \right\}.$$
The $\beta$-\textit{name} of a point $x\in X$ is the sequence $\beta(x)\in\left(\#\beta\right)^{\NN}$
defined by 
\[
\beta_{n}(x)=i\ \text{if and only if}\ T^{n}x\in\beta^{i}.
\]
The $\left(T,\beta,n\right)$ \textit{Hamming pseudo-metric} on $X$ is defined
by 
\[
\bar{d}_{n}^{\left(\beta\right)}(x,y)=\frac{1}{n}\#\left\{ k\in\left\{ 0,...,n-1\right\} :\ \beta_{k}(x)\neq\beta_{k}(y)\right\} .
\]
That is two points $x,y\in X$ are $\bar{d}_{n}^{\left(\beta\right)}$
close if for most of the $k's$ in $\{0,..,n-1\}$, $T^{k}x$ and
$T^{k}y$ lie in the same partition element of $\beta$. An $\epsilon$-ball in the  Hamming pseudo-metric will be denoted by 
\[
B\left(n,\beta,x,\epsilon\right):=\left\{ y\in X:\ \bar{d}_{n}^{\left(\beta\right)}(x,y)<\epsilon\right\} .
\]

This pseudo-metric was used in \cite{Katok Thouvenot,Ferenczi} to
define complexity sequences and slow-entropy-type invariants. It was
shown for example by Katok and Thouvenot \cite{Katok Thouvenot} that if the growth rate
of the complexity sequence is of order $e^{hn}$ with $h>0$, then
$h$ equals the entropy of $X$, by Ferenczi  \cite{Ferenczi} that $T$ is isomorphic
to a translation of a compact group if and only if the complexity
is of lesser order from any sequence which grows to infinity and
more. In this paper we will be interested with the relativised versions
of these invariants which were introduced in Aaronson \cite{Aaronson}. 

A $T$-invariant sub-$\sigma$-algebra $\CC\subset\BB$ is called
a {\it factor}. An equivalent definition in ergodic theory is a probability
preserving transformation $\big(Y,\tilde{\CC},\nu,S\big)$ with
a (measurable) factor map $\pi:X\to Y$ such that $\pi T=S\pi$ and
$\nu=m\circ\pi^{-1}$, in this case $\CC=\pi^{-1}\tilde{C}$. 

Given a factor $\CC\subset\BB$, $n\in\NN$, $\beta\in\mathfrak{B}(X)$ and $\epsilon>0$ we define
a $\CC$-measurable random variable $\mathcal{K}_{\CC}\left(\beta,n,\epsilon\right):X\to\NN$
by 
\[
\mathcal{K}_{\CC}\left(\beta,n,\epsilon\right)(x)
:=\min\Bigg\{ \#F:\ F\subset X,\ m\left(\left.\bigcup_{z\in F}B(n,\beta,z,\epsilon)\right| \CC\right)(x)>1-\epsilon\Bigg\} ,
\]
where $m\left(\left.\cdot\right|\CC\right)$ denotes the conditional
measure of $m$ with respect to $\CC$. The sequence of random variables
$\left\{ \mathcal{K}_{\CC}\left(\beta,n,\epsilon\right)\right\} _{n=1}^{\infty}$
is called the \textit{relative complexity} of $\left(T,\beta\right)$ with
respect to  $\beta$ given $\CC$. 

The {\it{upper entropy dimension of $T$ given $\CC$}} is defined by
\[
\overline{{\rm EDim}}\left(T,\CC\right):=\inf\left\{ t>0:\ \limsup_{n\to\infty}\frac{\log\mathcal{K}_{\CC}\left(\beta,n,\epsilon\right)}{n^{t}}=0,\ \forall\beta\in\mathfrak{B}(X)\right\} 
\]
and the {\it{lower entropy dimension of $T$ given $\CC$}} is
\[
\underline{{\rm Edim}}\left(T,\CC\right)=\sup\left\{ t>0:\ \exists\beta\in\mathfrak{B}(X),\ \liminf_{n\to\infty}\frac{\log\mathcal{K}_{\CC}\left(\beta,n,\epsilon\right)}{n^{t}}=\infty\right\} .
\]
In case $\underline{{\rm Edim}}\left(T,\CC\right)=\overline{{\rm Edim}}\left(T,\CC\right)=a$ we write ${\rm Edim}\left(T,\CC\right)=a$ and call this quantity the {\it{entropy dimension of $T$ given $\CC$}}. 

Given a sequence of random variables $Y_n,n\in\mathbb{N}$ taking values on $[0,\infty]$ we write $Y_n\xrightarrow[n\to\infty]{\mathfrak{D}}Y$ to denote ``$Y_n$ converges to $Y$ in distribution" and $Y_n\xrightarrow[n\to\infty]{m} Y$ to denote convergence in probability.
The next Theorem is a special case of \cite[Thm 2]{Aaronson} when $\left\{ n_k\right\}_{k=1}^{\infty}=\mathbb{N}$. 

\begin{theorem}[Aaronson's Generator Theorem]
\label{thm: Aaronson's RC} 
Let $\left( X,\mathcal{B},m,T\right)$ be a measure preserving transformation and a sequence $d_n>0$.
\begin{itemize}
\item[(a)] If there is a countable $T$-generator $\beta\in\mathfrak{B}(X)$ and a random variable $Y$ on $[0,\infty]$ satisfying
$$\frac{\log\mathcal{K}_{\CC}\left(\beta,n,\epsilon\right)}{d_n}\xrightarrow[n\to\infty,\ \epsilon\to 0]{\mathfrak{D}} Y$$
Then for all $T$-generating partitions $\alpha\in\mathfrak{B}(X)$, 
$$\frac{\log\mathcal{K}_{\CC}\left(\alpha,n,\epsilon\right)}{d_n}\xrightarrow[n\to\infty,\ \epsilon\to 0]{\mathfrak{D}} Y$$
\item[(b)] if for some $\beta\in \mathfrak{B}(X)$, a generating partition for $T$,
\[
\frac{\log\mathcal{K}_{\CC}\left(\beta,n,\epsilon\right)}{n^{t}}\xrightarrow[n\to\infty,\epsilon\to0]{m}0,
\]
then $\overline{{\rm EDim}}\left(T,\CC\right)\leq t$.

\item[(c)] if for some partition $\beta\in \mathfrak{B}(X)$, 
\[
\frac{\log\mathcal{K}_{\CC}\left(\beta,n,\epsilon\right)}{n^{t}}\xrightarrow[n\to\infty,\epsilon\to0]{m}\infty
\]
then $\underline{{\rm Edim}}\left(T,\CC\right)\geq t$.

\end{itemize}
\end{theorem}

\subsection{Basic ergodic theory for $\mathbb{Z}^d$ actions}
Let $\left(X,\BB,m\right)$ be a standard probability space and $G$ be an Abelian countable group. A measure preserving action of $G$ on $\left(X,\BB,m\right)$ is a map $\cS: G\to \mathrm{Aut} \left(X,\BB,m\right)$ such that for every $g_1,g_2\in G$, $\cS_{g_1 g_2}=\cS_{g_1}\cS_{g_2}$ and for all $g\in G$, $\left(\cS_{g}\right)_{*}m=m$. The action is \textit{ergodic} if there are no non trivial $\cS$-invariant sets.

Given an ergodic $G$ action $\left(X,\BB,m,\cS \right)$ and increasing sequence $F_n$ of subsets of $G$ one can define a sequence of averaging operators $A_n: L_2 (X,\BB,m)\circlearrowleft$ by
\[
A_n(f):=\frac{1}{\# F_n} \sum_{g\in F_n}f\circ \cS_g
\]
and ask whether for all $f\in L_2 (X,\BB,m)$ one has  $A_n(f)\to \int_X f dm$ in $L_2$. The sequences of sets $\left\{F_n\right\}_{n=1}^\infty$ for which this is necessarily true are called {\it F\"olner sequences} and they are characterised by the property that for every $g\in G$,
\[
\frac{\# \left[ F_n\triangle \left\{F_n+g\right\} \right]}{\# F_n }\xrightarrow[n\to\infty]{} 0.
\]
 
 In this work we will be concerned with either actions of $G=\mathbb{Z}$ which is generated by one measure preserving transformation or $G=\mathbb{Z}^2$ which corresponds to two commuting  measure preserving transformations. 
For a finite partition $\beta$ of $X$, one defines the \textit{entropy} of $\cS$ with respect to $\beta$ by
\[
h(\cS,\beta):=\lim_{n\to_\infty} \frac{1}{n^d}H\left( \bigvee_{j\in [0,n]^d \cap \mathbb{Z}^d} \cS_{j}^{-1}\beta \right),
\]
where $H(\beta)=\sum_{j=1}^{\#\beta} m\left( \beta^i \right)\log m\left( \beta^i\right)$ is the \textit{Shannon entropy} of the partition. The entropy of $\cS$ is then defined by
\[
h(\cS)=\sup_{\beta\in \mathfrak{B}(X):\ \beta\ \text{finite}} h(\cS,\beta)
\]
As in the case of a $\mathbb{Z}$ action, one says that $\beta$ is a generating partition if the smallest sigma algebra 
containing $\vee_{j\in\mathbb{Z}^d} \cS_{j}^{-1}\beta$ is $\mathcal{B}$. In an analogous way to the case of $\mathbb{Z}$ actions, it follows that if $\beta$ is a generating partition for $\cS$ then $h(\cS)=h(\cS,\beta)$ and if $h(\cS)<\infty$ then there exists finite generating partitions \cite{Krieger, Katsnelson Weiss, Danilenko Park}.

\section{Random walks in random sceneries and statement of main theorem}

In what follows we will be interested in a random walk in random scenery
where the jump random variable $\xi\in \mathbb{Z}^2$ is in the domain of attraction of 
$2$-dimensional Brownian Motion or $\xi \in \mathbb{Z}$ is strongly aperiodic and in the domain of attraction of the Cauchy law. The reason that these two
models are of most interest to us is that the limiting distribution
does not have a local time process.

To be more precise 
let $\xi, \xi_1, \xi_2, \dots$ be i.i.d. $\mathbb{Z}^d$-valued random variables defined on a probability space $(\Omega, \mathcal{F}, \mathbb{P})$, with characteristic function 
$\phi_\xi(t):=\mathbb{E}\big( e^{it\cdot\xi}\big)$ for $t\in [-\pi,\pi]^d$, 
and that either
\begin{description}
\item[A1] (\textit{$1$-stable}) $\xi \in \mathbb{Z}$ and $\phi_\xi(t)= 1- \gamma |t| + o(|t|)$ for $t\in [-\pi,\pi]$, for some $\gamma>0$; or
\item[A2] $\xi$ is in $\mathbb{Z}^2$ and $\rE |\xi|^2<\infty$ with non-singular covariance matrix $\Sigma$; equivalently
$\phi_\xi(t) = 1- \langle t, \Sigma t\rangle + o(|t|^2)$ for $t\in [-\pi,\pi]^2$.
\end{description} 

In the above cases the random walk $S_{n}\left(\xi\right):=\xi_{1}+\xi_{2}+\cdots+\xi_{n}$ we will also assume that the random walk
is \textit{strongly aperiodic} in the sense that there is no proper subgroup $L$ of $\mathbb{Z}^d$ such that $\mathbb{P}(\xi -x \in L)=1$ for some $x\in \mathbb{Z}^d$ . 

We are also interested in the two dimensional \textit{Simple Random Walk}, which has period 2 and is thus not covered by {\bf A2} above.
\begin{description}
\item[A2'] $\xi \in \mathbb{Z}^2$ and $\rP[\xi =e] = 1/4$ for $|e|=1$. Then $\sqrt{\det(\Sigma)}=1/2$.
\end{description}

Denote by $\mu_{\xi}$ the distribution of $\xi$. The base of the
RWRS is then defined as $\Omega=\big(\mathbb{Z}^{d}\big)^{\NN}$
the space of all $\mathbb{Z}^{d}$-valued sequences, $\mathbb{P}=\prod_{k=1}^{\infty}\mu_{\xi}$, 
the product measure, 
and $\sigma:\Omega\to\Omega$ the left shift on $\Omega$ defined
by 
\[
\left(\sigma w\right)_{n}=w_{n+1}.
\]
 When $d=2$, the random scenery is an ergodic probability preserving $\ZZ$-
action $\left(Y,\CC,\nu,\cS\right)$ and when $d=1$ it is just an ergodic
probability preserving transformation $\cS:\left(Y,\CC,\nu\right)\to\left(Y,\CC,\nu\right)$. 

The skew product transformation on $Z=\Omega\times Y$, $\BB_{Z}=\BB_{\Omega}\otimes\BB_{Y}$,
$m=\mathbb{P}\times\nu$, defined by 
\[
T(w,y)=\left(\sigma w,\cS_{w_{1}}(y)\right),
\]
is the \textit{random walk in random scenery} with scenery $\left(Y,\CC,\nu,\cS\right)$
and \textit{jump random variable} $\xi$. 
\begin{theorem}
\label{thm: relative complexity of RWRS}
Let $\left(Z,\BB_{Z},m,T\right)$
be RWRS with random scenery $\left(Y,\CC,\nu,\cS\right)$ and
 jump random variable $\xi$.
\begin{itemize}
\item[(a)] If $d=1$ and $\xi$ satisfies \textbf{A1} then for any generating
partition $\mathbf{\beta}$ for $T$, 
\[
\frac{\log(n)}{\pi \gamma n} \mathcal{K}_{\mathcal{B}_{\Omega}}\left(\beta,n,\epsilon\right)\xrightarrow{m}h\left(\cS\right).
\]
\item[(b)] If $d=2$ and $\xi$ satisfies \textbf{A2} or \textbf{A2'} then for any generating
partition $\mathbf{\beta}$ for $T$, 
\[
\frac{\log n}{2\pi \sqrt{\det (\Sigma)}n}\mathcal{K}_{\mathcal{B}_{\Omega}}\left(\beta,n,\epsilon\right)\xrightarrow{m}h\left(\cS\right).
\]
\end{itemize}
 In particular in both cases 
\[
{\rm Edim}\left(T,\BB_{\Omega}\right)=1.
\]

\end{theorem}

\begin{remark}
This Theorem states that the rate of growth of the complexity is of the order $\#R(n)$ where $R(n)$ is the range of the random walk up to time $n$. This conclusion is similar to the conclusion of Aaronson for the case where the random walk is in the domain of attraction of an $\alpha$-stable random variable with $1<\alpha\leq 2$. Our method of proof can apply to these cases as well. In addition since we are not using the full theory of weak convergence of local times one can hope that this method will apply also to a wider class of dependent jump distributions. 
\end{remark}

Two probability preserving transformations $\left( X_i,\mathcal{B}_i,m_i,T_i \right), 1=1,2,$ are relatively isomorphic over the factors $\mathcal{C}_i\subset \mathcal{B}_i$ if there exists a measurable isomorphism $\pi:\left( X_1,\mathcal{B}_1,m_1,T_1 \right)\to\left( X_2,\mathcal{B}_2,m_2,T_2 \right)$ such that $\pi^{-1}\mathcal{C}_2=\mathcal{C}_1$. The following corollary follows from Theorem \ref{thm: relative complexity of RWRS} together with \cite[Corollary 4] {Aaronson}. 

\begin{corollary}
Suppose that $\left(Z_i,\BB_{Z_i},m_i,T_i \right),\ i=1,2$, are two Random walks in random sceneries with
 strongly aperiodic $\mathbb{Z}^2$ valued jump random variable $\xi$ which satisfy \textbf{A2} and their sceneries $\cS^{(i)}$ have finite entropies.
 \\ If these two systems are isomorphic over their bases $\mathcal{B}_{\Omega_i}$ then 
\[
\sqrt{ \det\left( \Sigma_1 \right)}h(\cS ^{(1)})=\sqrt{\det\left( \Sigma_2\right)}h(\cS^{(2)}).
\] 
\end{corollary}

\section{The range of the random walk} \label{sec: range}
Let $R(n)= \{ S(1), \cdots , S(n)\}$, be the range of the random walk and for $x\in\mathbb{Z}^d$ define the local time, 
$$l(n,x) = \sum_{j=1}^n \mathbf{1}\{S(j) = x\}.$$ Denote by $\mathcal{F}$ the $\sigma$-algebra generated by $\left\{X_n\right\}_{n=1}^\infty$

%
The following theorem extends \cite[Theorem~2]{cerny07}.
\begin{theorem}\label{thm:distribution}
Let $Y_n$ be a point chosen uniformly at random from $R(n)$, that is
\begin{equation}\label{eq:yn}
\rP[ Y_n=x \mid \mathcal{F}] = \frac{ \mathbbm{1}\{ x \in R(n)\}}{\# R(n)}.
\end{equation}
\begin{enumerate}[\upshape(i)]
\item If \textnormal{\bf A1} holds, then
\begin{equation}
\mathbb{P} \Big[ \pi \gamma \frac{l(n, Y_n)}{\log n} \geq u \Big | \mathcal{F} \Big] \to \re^{-u}, \quad \text{a.s. as $n\to \infty$;}
\end{equation}
\item  If \textnormal{\bf A2}(\cite[Theorem~2]{cerny07}) or {\bf A2'} holds then
\begin{equation}
\mathbb{P} \Big[ 2 \pi \sqrt{\det(\Sigma)} \frac{l(n, Y_n)}{\log n} \geq u \Big | \mathcal{F} \Big] \to \re^{-u}, \quad \text{a.s. as $n\to \infty$;}
\end{equation}
\end{enumerate}
\end{theorem}
The following is the main result of this section.
\begin{theorem}
Suppose that \textnormal{\bf A1}, \textnormal{\bf A2} or {\bf A2'} holds, then $R(n)$ is almost surely a  F\"olner sequence, 
that is for all $w \in \mathbb{Z}^d$
\begin{equation}
\frac{\# \Big[R(n) \triangle \big(R(n)+w\big)\Big]}{\# R(n)} \to 0.
\label{eq:folner}
\end{equation}
\label{thm:folner}
\end{theorem}
\subsection{Proof of Theorem~\ref{thm:distribution}.}
First of all, recall that the result under {\bf A2} has been proven in \cite{cerny07}. We will therefore focus on the remaining cases.
{We write $C$ for a generic positive constant. }
\subsubsection{Auxiliary results}
Before we embark on the proof of Theorem~\ref{thm:distribution}, we require several standard results. 

The next result is a direct consequence of strong aperiodicity and Assumptions \textbf{A1} and \textbf{A2}. 
Its proof is a standard application of Fourier inversion, and is included in the Appendix for the sake of completeness.
\begin{lemma}
Suppose that \textbf{A1} or \textbf{A2} holds. 
Then with $\gamma_1:= \pi \gamma$ and $\gamma_2:= 2\pi\sqrt{|\Sigma|}$
\begin{align}
\sup_{w}\rP[S(m)=w]&=O\left(\frac{1}{m}\right), \label{lclt_upperbound} \\
\rP[S(m)=w] - \rP[S(m)=0] &= O\left(\frac{|w|}{m^2}\right)\label{eq:potentialbound}\\
\rP[S(m)=w]&\sim \frac{1}{\gamma_d m}\label{lclt_exact}.
\end{align}
\label{lem:llt}
\end{lemma}
\begin{lemma}Suppose that \textbf{A1} holds. Then  as $\lambda \uparrow 1$
\label{lem:lambdaasymptotic}
\begin{equation}
\frac{1}{2\pi} \int_{-\pi}^\pi \frac{\lambda \phi(t) \rd t}{1-\lambda \phi(t)}
\sim \frac{1}{\pi \gamma} \log\Big( \frac{1}{1-\lambda} \Big).
\label{eq:lambdaasymptotic}
\end{equation}
\end{lemma}
Since simple random walk is not aperiodic, to prove Theorem~\ref{thm:distribution} for the case {\bf A2'} we recall the following (see \cite[Theorem~1.2.1]{Lawler91}).
\begin{lemma}\label{lem:SRWlclt}
Under {\bf A2'} 
\begin{align}
\sup_x\rP[S(m)=x] = O\Big(\frac{1}{m}\Big),\label{eq:SRWlclt}\\
\sum_{k=0}^n \rP[S_m=0] \sim \frac{1}{\pi } \log n\label{eq:SRWgreen}
\end{align}
\end{lemma}

For $\alpha \geq 0$ we define the $\alpha$-\textit{fold self-intersection local time}
\begin{align*}
L_{n}(\alpha) 
&:= \sum_{x \in \mathbb{Z}^d} l(n,x)^{\alpha}, \qquad \alpha >0\\
L_n(0) &:=\lim_{\alpha \downarrow 0} L_n(\alpha)= \sum_{x \in \mathbb{Z}^d} \mathbf{I}\{ l(n,x)>0\} =\# R(n).
\end{align*}
We will need the following strong law of large numbers which is given in \cite{cerny07} for the case {\bf A2}.
The case {\bf A2'} is included in order to demonstrate how one can handle the periodic case.
\begin{proposition}
For $d=1, 2$, and any integer $k \geq 1$ if \textnormal{\bf A1} or \textnormal{\bf A2'}  holds then as $n \to \infty$
\begin{align}
\rE L_n (k) 
&\sim \frac{\Gamma(k+1)}{(\pi \gamma_d)^{k-1}} n (\log n)^{k-1}, \label{eq:expasymptotic}\\
\var( L_n(k))
&= O\big( n^2 (\log n)^{2k -4}\big)\label{eq:varasymptotic},\\
\lim_{n\to\infty}\frac{n (\log n)^{k-1}}{(\pi \gamma_d)^{k-1}} L_n(k) &=
\Gamma(k+1), \qquad \text{almost surely}\label{eq:lln1d}.
\end{align}
\label{propn:lln}
\end{proposition}
\begin{proof}[Proof of Proposition~\ref{propn:lln}]
Once \eqref{eq:expasymptotic} and \eqref{eq:varasymptotic} have been established, 
\eqref{eq:lln1d} follows for geometric subsequences by Chebyshev's inequality, and the complete result by the same argument as in \v{C}erny \cite{cerny07}.

\smallskip
\noindent{\textsl{Case {\bf A1}:}} The estimate \eqref{eq:varasymptotic} is contained in Theorem~3 of Deligiannidis and Utev \cite{DU15}.  It remains to prove \eqref{eq:expasymptotic}.

Similar to \cite{cerny07}, we write 
\begin{align*}
\rE L_n (k)
&= \sum_{j_1, \dots, j_k=0}^n \rP [ S_{k_1} = \dots = S_{j_k}]
= \sum_{b=1}^k \rho(b,k) \sum_{0\leq j_1 < \dots < j_b \leq n}
\rP[S_{j_1} = \dots = S_{j_b}],
\end{align*}
where $\rho(k,k) = k!$, while the remaining factors will not be important.

Letting 
\begin{equation}
M_n(b):= \{ (m_0, \dots, m_b)\in \mathbb{N}^{b+1}: m_1, \dots, m_{b-1} \geq 1, \sum m_i =n\},
\label{eq:Mn}
\end{equation}
we have by the Markov property
\begin{equation}\label{eq:abn}
a_b(n):=\sum_{0\leq j_1 < \dots < j_b \leq n}
\rP[S_{j_1} = \dots = S_{j_b}] = \sum_{m \in M_n(b)} \prod_{i=1}^{b-1} \rP [S_{m_i}=0].
\end{equation}
Then for  $\lambda \in [0,1)$, by standard Fourier inversion
\begin{align*}
\sum_{n=0}^\infty a_b(n) \lambda ^n  
&= \sum_{n=0}\lambda^n \sum_{m \in M_n(b)} \prod_{i=1}^{b-1} \rP [S_{m_i}=0]\\
&= \sum_{m_0\geq 0}\sum_{m_1, \dots, m_{b-1}\geq 1}^\infty \sum_{n=0}^\infty 
\lambda^{m_0+\dots +m_{b-1}+n} \prod_{i=1}^{b-1} \rP [S_{m_i}=0]\\
&= \sum_{m_0=0}^\infty \lambda ^{m_0} \sum_{n=0}^\infty \lambda^n 
 \prod_{i=1}^{b-1} \sum_{m_i=1}^\infty \lambda^{m_i}\rP [S_{m_i}=0]\\
&= \frac{1}{(1-\lambda)^{2}} \bigg[\frac{1}{2\pi} \int_{-\pi}^\pi \frac{\lambda \phi(t) \rd t}{1-\lambda \phi(t)}\bigg]^{b-1}
\sim \frac{(\pi \gamma)^{1-b}}{(1-\lambda)^{2}} \log\Big( \frac{1}{1-\lambda}\Big)^{b-1},
\end{align*}
as $\lambda \uparrow 1$, by Lemma~\ref{lem:lambdaasymptotic}.
Then under {\bf A1} \eqref{eq:expasymptotic} follows by Karamata's Tauberian theorem, since the sequence 
$a_b(n)$ is monotone increasing. 

\smallskip
\noindent\textsl{Case {\bf A2'}:}
The estimate \eqref{eq:expasymptotic} follows from \eqref{eq:abn} and \eqref{eq:SRWgreen}.

{The proof of \eqref{eq:varasymptotic} can be adapted from \cite{DU15}.}
The variance is given by
\begin{align*}
\var(L_n(k))
&= C(k)\sum_{i_1 \leq \cdots \leq i_k} \sum_{l_1\leq \dots \leq l_k}
\Big\{ \rP \big[ S(i_1) = \cdots = S(i_k); S(l_1) = \cdots = S(l_k)\big]\\
&\qquad -
\rP \big[ S(i_1) = \cdots = S(i_k)\big] \rP\big[S(l_1) = \cdots = S(l_k)\big]
\end{align*}
The terms where $l_1, \dots, l_k$ is not completely contained in any of the intervals 
$[i_j, i_{j+1}]$ can be bounded above by the positive term in the sum using 
\eqref{eq:SRWlclt} and the approach in \cite{DU15}. A similar, albeit more involved, calculation is performed in the proof of Proposition~\ref{lem:lln}.

Suppose then that $l_1,\dots, l_k \in [i_j, i_{j+1}]$ for some $j$, and by symmetry we can take $j=1$. 
Define $M_n(2k)$ as in \eqref{eq:Mn} and change variables to 
\begin{gather*}
i_1=m_0,\,\,\, l_1= m_0 +m_1, \,\,\,l_2 = m_0+m_1+m_2, \dots,\,l_k = m_0+\cdots+m_{k}\\
i_2 = l_k + m_{k +1},\,\,\, \cdots, i_{k} = l_k + m_{k+1} + \cdots
+m_{2k -1}.
\end{gather*}
Write $p(m) = \rP[S(m)=0]$ and 
$\bar{p}(m) = 1/(\pi m)$. The contribution of these terms is then 
\begin{align*}
{J_n(k)}
&=C(k)
\sum_{M_n(2k)}\prod_{\substack{1\leq j \leq 2k -1\\j\neq 1, k+1}}
\!\!\!\! p(m_j)
\times \Big\{
p(m_1+m_{k +1}) 
 - p(m_1 + \cdots +m_{k+1})\Big\}.
\end{align*}
By \cite[Theorem~2.1.3]{LawlerLimic} we have that
$$|p(m)+p(m+1)-2\bar{p}(m)| \leq \frac{C}{m^2}.$$
Let $q:= m_2+\cdots +m_k$ and 
$$M := n-\sum_{\substack{0\leq j \leq 2k -1\\j\neq 1, k +1}}m_j.$$
Then 
\begin{align*}
\lefteqn{\Big|\sum_{m_1+m_{k+1}=0}^{M} p(m_1+m_{k+1}) - p(m_1+m_{k+1}+q) \Big|}\\
&\leq \sum_{m_1=0}^{M}
\sum_{m_{k+1}=0}^{[(M-m_1)/2]}\Big| p(m_1+2m_{k+1}) + p(m_1+2m_{k+1}+1) \\
&\qquad -p(m_1+2m_{k+1}+q)-p(m_1+2m_{k+1}+1 +q) \Big|\\
&\leq \sum_{m_1=0}^{M}\sum_{m_{k+1}=0}^{[(M-m_1)/2]}
\Big\{ 
|\bar{p}(m_1+2m_{k+1}) - \bar{p}(m_1+2m_{k+1}+q)| +
\frac{C}{(m_1+2m_{k+1})^2}\Big\}\\
&\leq \sum_{m_1=0}^{M}\sum_{m_{k+1}=0}^{[(M-m_1)/2]}
\Big\{ 
\frac{q}{(m_1+2m_{k+1})(m_1+2m_{k+1}+q)}
+\frac{C}{(m_1+2m_{k+1})^2}\Big\}\\
&\leq \sum_{m_1+m_{k+1}=0}^{n}
\Big\{ 
\frac{q}{(m_1+m_{k+1})(m_1+m_{k+1}+q)}
+\frac{C}{(m_1+m_{k+1})^2}\Big\}.
\end{align*}
Thus going back to $J_n(k)$ we have
\begin{align*}
{J_n(k)}
&\leq 
\sum_{M_n(2k)}\prod_{\substack{1\leq j \leq 2k -1\\k\neq 1, k+1}}
\!\!\!\! p(m_k)
\times \Big\{
\frac{m_2+\cdots+ m_k}{(m_1+\cdots+ m_{k +1})(m_1+m_{k+1}) }
+ \frac{C}{(m_1+m_{k +1})^2}  \Big\}\\
&=\sum_{M_n(2k)}\prod_{\substack{1\leq k \leq 2k -1\\k\neq 1, k+1}}
\!\!\!\! p(m_k)
\frac{m_2+\cdots+ m_k}{(m_1+\cdots+ m_{k +1})(m_1+m_{k+1}) } + O\big(n (\log n)^{2k -2} \big)\\
&=: J'_n(k)+ O\big(n (\log n)^{2k -2} \big).
\end{align*}
By symmetry after we split the sum in the numerator and we combine 
$m=m_1+m_{k +1}$ 
\begin{align*}
J'_n(k)
&\leq C k n \sum_{m_1, \dots, m_{2k-1}=0}^n
\frac{1}{m_3 \cdots m_{2k-1} (m_1+m_{k+1}) (m_1+\cdots+m_{k+1}) }\\
&\leq Cn (\log n)^{2k -4} \sum_{m, m_2=0}^n \frac{1}{m+m_2} \leq Cn^2 (\log n)^{2k -4}.\qedhere
\end{align*}
\end{proof}
\begin{remark}
A similar proof can be performed for any periodic random walk, by summing over the period.
\end{remark}
\subsubsection{Proof of Theorem~\ref{thm:distribution}}
\begin{proof}[{Proof of Theorem~\ref{thm:distribution}}]
The proof is very similar to \cite{cerny07} and the end of Theorem~\ref{thm:folner} and is thus ommitted. We just point out
that under {\bf A1}
\begin{equation} \frac{\log(n)}{\pi \gamma n}  \# R(n)\to 1, \qquad \text{a.s. as $n\to \infty$}, \label{Erdos Dvoretsky for Cauchy}\end{equation}
by a simple application of Result 2 in Le Gall and Rosen~\cite{lGallRosen91} with $\beta=d=1$ and $s(n)\equiv 1$, after one notices that in our case the truncated \textsl{Green's function} satisfies
\begin{equation}
\label{eq:green}
h(n) := \sum_{k=0}^n  \rP(S_k=0) \sim \frac{\log(n)}{\pi \gamma},
\end{equation}
by Lemma~\ref{eq:lambdaasymptotic} and Karamata's Tauberian theorem.

For {\bf A2'} note that \cite[Theorem~4]{Dvoretzky Erdos} states that
\begin{equation*}
\frac{\log n}{\pi n}\# R(n) \to 1.\qedhere
\end{equation*}
\end{proof}
\subsection{Proof of the F\"olner property of the Range (Theorem~\ref{thm:folner})}
Let $\alpha > 0$ and define
$$L_{n,w}(\alpha) 
:= \sum_{x \in \mathbb{Z}^d} l(n,x)^{\alpha} l(n,x+w)^{\alpha}.$$
These quantities are of interest since
\begin{align*}
L_{n,w}(0) 
&:=\lim_{\alpha \downarrow 0} L_{n,w}(\alpha)
= \sum_{x\in \mathbb{Z}^d} \II\Big(l(n,x)>0 \Big) \II\Big(l(n,x+w)>0\Big)\\
&= \# (R(n)\cap R(n)+w).
\end{align*}
Using the above notation the F\"olner property~\eqref{eq:folner} can be written as 
\[\lim_{n\to \infty} \frac{L_{n,w}(0)}{L_{n,0}(0)} = 1.\qedhere
\] 
We will use the following result.
\begin{proposition}
Assume {\bf A1} or {\bf A2} holds. For all $w\in \mathbb{Z}^d$ and $\alpha \in \mathbb{Z}$, $\alpha \geq 1$
$$\frac{L_{n,w}(\alpha) }{n (\log n)^{2\alpha -1}} \to 
\begin{cases}
\frac{\Gamma(2\alpha+1)}{(\pi \gamma)^{2\alpha -1}}, & \mbox{for $d=1$}\\
\frac{\Gamma(2\alpha+1)}{(2\pi \sqrt{|\Sigma|})^{2\alpha -1}}, & \mbox{for $d=2$.}
\end{cases}$$\label{lem:lln}
\end{proposition}
We first complete the proof of Theorem~\ref{thm:folner} and then we will prove the above Proposition.
\begin{proof}[Proof of Theorem~\ref{thm:folner}]
We first treat the cases {\bf A1} and {\bf A2}.

Let $Y_n$ be defined as in \eqref{eq:yn}. 
Setting $\gamma_d = 2\pi \sqrt{|\Sigma|}$ for $d=2$ and $\gamma_d = \pi \gamma$ for $d=1$, 
define 
$$Z_n:= \gamma_d^2 \frac{l(n,Y_n) l(n,Y_n+w)}{\log(n)^2}.$$
For integer $\alpha$, by Proposition~\ref{lem:lln}
\begin{align*}
\rE[ Z_n^\alpha | \mathcal{F} ]
&= \frac{\gamma_d^{2\alpha} }{\#R(n)} \sum_x \frac{l(n,x)^\alpha l(n,x+w)^\alpha }{\log(n)^{2\alpha}}\\
&= \frac{\gamma_d^{2\alpha -1} L_{n,w}(\alpha) }{n \log(n)^{2\alpha -1}} 
	\frac{\gamma_d n/\log(n)}{R(n)} \to \Gamma(2\alpha +1).
\end{align*}
These are the moments of $Y^2$, where $Y \sim \mathrm{Exp}(1)$. 
Since 
$$\limsup_{k\to\infty} \frac{\Gamma(1+2k)^{1/2k}}{2k}=\lim \frac{\Gamma(1+2k)^{1/2k}}{2k} = \re^{-1} <\infty,$$
these moments define a unique distribution on the positive real line (see \cite{Durr10}),
and therefore $\mathbb{P}$-almost surely, we have that conditionally on $\mathcal{F}$, 
$Z_n \to Y^2$ in distribution. 
Then
\begin{align*}
\frac{\sum_x \II(l(n,x)>0, l(n,x+w) >0)}{\# R(n)} 
&= \lim_{\alpha \downarrow 0} 
\frac{\gamma_d^{2\alpha} }{R(n)} \sum_x \frac{l(n,x)^\alpha l(n,x+w)^\alpha }{\log(n)^{2\alpha}}\\
&=\lim_{\alpha \downarrow 0} \rE [Z_n^\alpha|\mathcal{F}], \quad \text{and by monotone convergence} \\
&=\rE\Big[ \lim_{\alpha \to 0}Z_n^\alpha \Big|\mathcal{F}\Big]
= \rP(Z_n >0 |\mathcal{F}).
\end{align*}
This shows that 
\begin{equation*}
\frac{\sum_x \II(l(n,x)>0, l(n,x+w) >0)}{\# R(n)} 
=\rP(Z_n>0|\mathcal{F}) \xrightarrow[n\to\infty]{} \rP(Y^2 >0) = 1.
\end{equation*}
%
%

\noindent \textsl{Simple Random Walk.} For the simple random walk in $\mathbb{Z}^2$ notice that one can consider the 
\textsl{lazy} version of the random walk, where 
$\rP[\xi'=0]=1/2$ while for $e \in \mathbb{Z}^2$, with $|e|=1$ we have
$\rP[\xi'=e]=1/4d$. Then the lazy simple random walk $S'_n:= \sum_{i=1}^n \xi'$, 
is strongly aperiodic and satisfies {\bf A2'} and therefore letting 
$R'(n):= \{ S'(0), \dots, S'(n)\}$ be the range of $\{S'(n)\}_n$ we have for all $w \in \mathbb{Z}^2$
$$\frac{\# (R'(n) \cap R'(n)+w)}{\#R'(n)} \to 1,$$
almost surely. 
Define recursively the successive jump times
$$T_0:= \min\{j\geq 1: S_j'\neq S'_{j-1}\}, \quad T_k := \min\{j> T_{k-1}: S'_j \neq S'_{j-1} \}.$$
Notice that the range of the simple random walk $R(n)$ is equal to the range of the lazy walk at the time of the $n$-th jump, $R'(T_n)$.  
Therefore 
$$\frac{\# (R(n) \cap R(n)+w)}{\#R(n)} = \frac{\# (R'(T_n) \cap R'(T_n)+w)}{\#R'(T_n)} \to 1,$$
since $T_n \to \infty$ almost surely.
\end{proof}
\begin{remark}
Note that it is also possible to prove Theorem~\ref{thm:folner} under {\bf A2'} directly, by proving the corresponding version of Proposition~\ref{lem:lln} and then following the same argument as for {\bf A2}. To adapt the variance calculation in Proposition~\ref{lem:lln} to the simple random walk, one has to sum first over the period similarly to the proof of Proposition~\ref{propn:lln}.
\end{remark}

\begin{proof}[Proof of Proposition~\ref{lem:lln}]
First we prove the result for $\alpha \in \mathbb{N}$ and then we extend it to the general case $\alpha \geq 0$.
For $\alpha \in \mathbb{N}$, we have
\begin{align*}
L_{n,w}(\alpha) 
&= \sum_{x \in \mathbb{Z}^2} \Big(\sum_{i=0}^n \II(S_i=x) \Big)^\alpha
	\Big(\sum_{i=0}^n \II(S_i=x+w) \Big)^\alpha\\
&= \sum_{x \in \mathbb{Z}^2} \sum_{i_1,\cdots, i_\alpha=0}^n \II\Big[S({i_1})=\cdots=S({i_\alpha})=x\Big]
	\sum_{k_1,\cdots, k_\alpha=0}^n \II\Big[S({k_1})=\cdots=S({k_\alpha})=x+w\Big]\\
&=\sum_{i_1,\cdots, i_{2\alpha}=0}^n \II\Big\{S({i_1})=\cdots=S(i_\alpha)=S(i_{\alpha+1})-w=\cdots=S(i_{2\alpha})-w\Big\},
\end{align*}							
which for $w=0$ is corresponds to the term $L_n(2\alpha)$. 
Then we can rewrite $L_{n,w}(\alpha)$ as
\begin{equation}\label{eq:Lnwa}
\begin{split}
L_{n,w}(\alpha) 
&=\sum_{\beta=1}^{2\alpha}
\sum_{j=(\beta-\alpha)\vee 0}^{\alpha\wedge \beta}
\sum_{\pmb{\epsilon} \in E(\beta, j)}j!(\beta-j)!\\
&\quad 
\sum_{0\leq i_1< \cdots < i_{\beta}\leq n} \II\Big\{S({i_1})+\epsilon_1 w=\cdots=\cdots=S(i_{\beta})+\epsilon_\beta w\Big\},
\end{split}
\end{equation}		
where the third sum is over the set
$$E(\beta, j):= \{ \pmb{\epsilon}=(\epsilon_1, \cdots, \epsilon_\beta) \in \{-1,0\}^{\beta}
: \sum |\epsilon_i| =j \}.$$

\medskip
\noindent\textbf{Expectation of $L_{n,w}(\alpha)$.}
For given $\beta$, $n$ and $\pmb\epsilon \in E(\beta,j)$, we have using the Markov property
\begin{align*}
\alpha(\pmb{\epsilon}, \beta, n)
&:=\rE\sum_{0\leq i_1< \cdots < i_{\beta}\leq n} \II\Big\{S({i_1})+\epsilon_1 w=\cdots=\cdots=S(i_{\beta})+\epsilon_\beta w\Big\}\\
&= \sum_{m \in M_n(\beta)}\prod_{i=1}^{\beta-1} \rP[S({m_i})= (\epsilon_i-\epsilon_{i+1})w],
\end{align*}						
Next we show that the asymptotic behaviour does not actually depend on $w$ or $\pmb\epsilon$. 
In this direction we rewrite  
\begin{align*}
\alpha(\pmb{\epsilon}, \beta, n)
&= \sum_{m \in M_n}\prod_{i=1}^{\beta-1} \rP[S({m_i})= 0]
 + \sum_{m \in M_n} \Big\{
\prod_{i=1}^{\beta-1}\rP[S({m_i})= (\epsilon_i-\epsilon_{i+1})w] - \prod_{i=1}^{\beta-1}\rP[S({m_i})= 0]\Big\}\\
&=:\alpha(\pmb{0}, \beta, n) + \mathcal{E}(\pmb{\epsilon}, \beta, n, w),
\end{align*}						
and we claim that $\mathcal{E}(\beta, n, w) = o(\alpha(\pmb{0}, \beta, n))$ as $n\to \infty$.

Letting $\delta_i = \epsilon_i- \epsilon_{i+1}$ we telescope the product to get
\begin{align}
\lefteqn{\mathcal{E}(\pmb\epsilon, \beta, n, w)
= \sum_{m \in M_n} 
\prod_{i=1}^{\beta-1}\rP[S_{m_i}= \delta_i w] - \prod_{i=1}^{\beta-1}\rP[S_{m_i}= 0]}\label{eq:SRWexpectation}\\
&=\sum_{j=0}^{\beta-1}\sum_{m \in M_n} 
\prod_{i=1}^{\beta-1-j}\rP[S({m_i})= \delta_i w]\times \Big[\rP[S({m_{\beta-1-j+1}})= \delta_{\beta-1-j+1} w] - 
\rP[S({m_{\beta-1-j+1}})= 0]\Big]\notag\\
&\qquad \times\prod_{l=\beta-1-j+2}^{\beta-1} \rP[S({m_l})= 0],\notag
\end{align}
where implicitly the indices are not allowed to exceed their corresponding ranges.

We analyse the first term in detail
\begin{align*}
\lefteqn{\bigg|\sum_{m \in M_n} \prod_{i=1}^{\beta-2} \rP[S({m_i})=\delta_i w]\times \Big[ \rP[S({m_{\beta-1}})=\delta_{\beta-1} w] - \rP[S({m_{\beta-1}})=0] \Big]\bigg|}\\
&\leq \sum_{m_0, \dots, m_{\beta-2}=0}^n \prod_{i=1}^{\beta-2} \rP[S({m_i})=\delta_i w]\times \sum_{m_{\beta-1}=0}^n \bigg| \rP[S({m_{\beta-1}})=\delta_{\beta-1} w] - \rP[S({m_{\beta-1}})=0]\bigg|\\
&\leq  \sum_{m_0, \dots, m_{\beta-2}0}^n \prod_{i=1}^{\beta-2}\rP[S({m_i})=\delta_i w]  \Big[1+\sum_{m_{\beta-1}=1}^\infty \frac{C}{m^2}\Big] 
= O\big( n \log(n)^{\beta -2}\big),
\end{align*}
by Lemma~\ref{lem:llt}, the remaining errors being very similar.

The asymptotic behaviour of $\alpha(\pmb{0}, \beta,n)$ follows from \cite{cerny07} for $d=2$ and Lemma~\ref{lem:lambdaasymptotic} for $d=1$ and is given by
\begin{align*}
\alpha(\pmb{0}, \beta, n)
&\sim n \Big( \frac{\log n}{\gamma_d} \Big)^{\beta -1}, 
\end{align*}
where $\gamma_1:=\pi \gamma$ and $\gamma_2:= 2\pi \sqrt{\det(\Sigma)}$.
Going back to \eqref{eq:Lnwa} we see that the leading term is for $\beta=2\alpha$, and from the above discussion, we can replace 
all terms $\alpha(\pmb{\epsilon},2\alpha,n)$ by $\alpha(\pmb{0},2\alpha,n)$. 
Since $\# E(2\alpha, \alpha) (\alpha!)^2 = \Gamma(2\alpha+1)$  we conclude that
\begin{align*}
\rE L_{n,w}(\alpha)
&= \rE \sum_{x\in \mathbb{Z}^d} l(n,x)^\alpha l(n,x+w)^\alpha
\sim \Gamma(2\alpha +1) \, n \Big( \frac{\log n}{\gamma_d}\Big)^{2\alpha -1}. 
\end{align*}
\medskip
\noindent\textbf{Variance of $L_{n,w}(\alpha)$.}
To compute the variance we will follow the approach developed in \cite{DU15}. 
First notice that
\begin{align*}
\rE L_{n,w}(\alpha)^2
&= \rE  \sum_{i_1, \dots, i_{2\alpha}=0}^n 
\mathbf{I} \Big( S(i_1)=\cdots=S(i_{\alpha})=S(i_{\alpha+1})-w=\cdots =S(i_{2\alpha})-w\Big)\\
&\qquad \times \sum_{j_1, \dots, j_{2\alpha}=0}^n 
\mathbf{I} \Big( S(j_1)=\cdots=S(j_{\alpha})=S(j_{\alpha+1})-w=\cdots =S(j_{2\alpha})-w\Big)
\end{align*}
Let $A_m, A_m'$ be 0 or 1 according to whether there is a $w$ or not in the $m$-th increment. Then
\begin{align*}
\var\big(L_{n,w}(\alpha)\big)
&= \sum_{k_1, \dots, k_{2\alpha}} \sum_{l_1, \dots, l_{2\alpha}}
	\bigg\{\rP\Big[S(k_1) = S(k_2) + A_2 w = \cdots =S(k_{2\alpha})+A_{2\alpha}w;\\
	&\qquad\qquad\qquad \qquad	S(l_1) = S(l_2) + A_2' w = \cdots =S(l_{2\alpha})+A_{2\alpha}'w\Big]\\
&\qquad - \rP\Big[S(k_1) = S(k_2) + A_2 w = \cdots =S(k_{2\alpha})+A_{2\alpha}w\Big]\\
&\qquad\qquad\qquad \qquad\qquad \times
	\rP\Big[S(l_1) = S(l_2) + A_2' w = \cdots =S(l_{2\alpha})+A_{2\alpha}' w\Big]\bigg\}.
\end{align*}

As we shall see the presence of $w$ does not affect the asymptotic. 
The main role is played by the interlacement of the sequences $\mathbf{k}=(k_1, \dots, k_{2\alpha})$ and $\mathbf{l} = (l_1, \dots, l_{2\alpha})$. 
In order to define the interlacement index $v(\mathbf{k}, \mathbf{l})$, of two sequences $\mathbf{k} = (k_1, \dots, k_r)$ and 
$\mathbf{l} = (l_1, \dots, l_s)$, let 
$\mathbf{j}$ be the combined sequence of length $r+s$, where ties between elements of $\mathbf{k}$ and $\mathbf{l}$ are counted twice. 
We also define $\pmb{\epsilon} = (\epsilon_1, \dots, \epsilon_{r+s})$, where $\epsilon_i = 1$ if the $i$-th element of the combined sequence is from $\mathbf{k}$ and 0 if it is from $\mathbf{l}$; that is ${j}_i \in \mathbf{k}$ and 0 otherwise. Then we define the \emph{interlacement index}, 
\begin{equation}
v(\mathbf{k}, \mathbf{l}) = v(k_1, \dots, k_r; l_1, \dots, l_s) := \sum_{i=1}^{r+s-1} |\epsilon_{i+1}-\epsilon_i|,
\end{equation}
which counts the number of times $\mathbf{k}$ and $\mathbf{l}$ cross over.

When $v=1$ then the contribution is zero by the Markov property. 
The main contribution will be from $v=2$. 
Similar to \cite{DU15}, the contributions of terms with $v \geq 3$ can be bounded above by just considering 
the positive part, $\rE L_{n,w}(\alpha)^2$. 
Let us first treat this case leaving $v=2$ for later.

\medskip
\noindent\emph{Case $v\geq 3$. }
Letting $\rho(\alpha)$ denote combinatorial factors, the contribution to $\rE L_{n,w}(\alpha)^2$ from the terms with interlacement $v\geq 3$ is trivially bounded above by
\begin{align*}
I_n(w, \alpha)
&:=\rho(\alpha)\sum_{k_1, \dots, k_{2\alpha}} \sum_{l_1, \dots, l_{2\alpha}}
	\bigg\{\rP\Big[S(k_1) = S(k_2) + A_2 w = \cdots =S(k_{2\alpha})+A_{2\alpha}w;\\
	&\qquad\qquad\qquad \qquad	S(l_1) = S(l_2) + A_2' w = \cdots =S(l_{2\alpha})+A_{2\alpha}'w\Big]\\
&= \rho(\alpha) \sum_{k_1, \dots, k_{2\alpha}} \sum_{l_1, \dots, l_{2\alpha}}
		\sum_{x} \rP\Big[S(k_1) =\cdots =S(k_{2\alpha})+A_{2\alpha}w;\\
	&\qquad\qquad\qquad\qquad S({l_1}) = S({k_1})-x,	S(l_1) = S(l_2) + A_2' w = \cdots =S(l_{2\alpha})		
		+A_{2\alpha}'w	\Big],
\end{align*}
where $A_i, A_i' \in \mathbb{Z}$ and may vary from line to line. 
Let $(j_1, \dots, j_{4\alpha})$ denote the combined sequence, allowing for matches. 
Changing variables 
$$j_1= m_0,\,\, j_2= m_0+m_1, \dots,\,\, j_{4\alpha} = m_0+\cdots+m_{4\alpha -1},\,\, n=m_0 + \cdots + m_{4\alpha},$$
with $m_0, \dots, m_{4\alpha} \geq 0$, we get
\begin{align*}
I_n(w, \alpha)
&\leq \rho(\alpha)\sum_{m_0, \dots, m_{4\alpha-1} \geq 0 }
	\sum_{x}\rP\Big[S(m_1) = S(m_1+m_2) + A_2 w + \delta_2 x = \\
	&\qquad\qquad\qquad \qquad\qquad\qquad\cdots =S(m_1+ \cdots +m_{4\alpha}) + A_{4\alpha} w + \delta_{4\alpha}x \Big]
\end{align*}
where $\delta_i:=\epsilon_i-\epsilon_{i+1} \in \{-1, 0, +1\}$, and $\pmb{\epsilon}$ is defined as earlier.
A simple application of the Markov property results in 
\begin{align*}
I_n(w,\alpha)&\leq  \rho(\alpha) n\sum_{m_1, \dots, m_{4\alpha -1}\geq 0} \sum_x \prod_{k=1}^{4\alpha -1}
\rP[ S({m_k}) = (\delta_{k-1}-\delta_k) x + A_k w],
\end{align*}
where the factor $n$ resulted from the free index $m_0$. Notice that since $v$ is the number of interlacements, exactly $u:= 4\alpha -1 -v$ of the $\delta$'s are 0 and thus by \eqref{eq:green}
\begin{align}
I_n(w,\alpha)
&\leq C n \log(n)^{4\alpha -1 -v} \sum_{j_1, \dots, j_v} \sum_x \prod_{t=1}^v \rP[S({j_t}) = \delta'_t x + A_t w]\label{eq:SRWvar1}
\end{align}
where $\delta'_t \in \{-1, +1\}$. Letting 
$$D_{n,v}:= \sum_{j_1, \dots, j_v=0}^n 
	\sum_x \prod_{k=1}^v \rP[S({j_k}) = \delta_k' x + A_k w],$$
notice that
\begin{align*}
D_{n,v} 
&\leq D_{n,v-1}  \sum_{j_v} \sup_y \rP\big[S({j_v})=y\big]
\leq C D_{n,v-1}  \sum_{j_v=1}^n \frac{1}{j_v}
\leq C\log(n) D_{n,v-1}.
\end{align*}
Repeating we arrive at 
$D_{n,v} 
\leq C \log(n)^{v-3} D_{n,3}$, and therefore
$$I_n(w, \alpha) \leq C n \log(n)^{4\alpha -4} D_{n,3}.$$
To complete our study of the $v\geq 3$ case we now treat the term $D_{n,3}$. 
\begin{align*}
D_{n,3}
&\leq C \sum_{i\leq j \leq k} \sum_x \rP[S(i) = \delta'_i x+ A_i w]\times
\rP[S(j) = \delta'_j x+ A_j w]\times\rP[S(k) = \delta'_k x+ A_k w]\\
&\leq C \sum_{i\leq j \leq k} \Big(\sup_y \rP[S(j) = y ]\Big) 
\sup_y\rP[S(i+k)=y],
\end{align*}
where $\tilde{S}_k$ denotes an independent copy of $S_k$. 
By symmetry and Lemma~\ref{lem:llt}
\begin{align*}
D_{n,3}
&\leq C \sum_{0\leq i\leq j \leq k\leq n} \frac{1}{j} \frac{1}{i+k}
\leq C \sum_{m_1, m_2, m_3=0}^n \frac{1}{m_1+m_2} \frac{1}{2m_1+m_2+m_3}\\
&\leq C \sum_{m_1, m_2=0}^n \frac{1}{m_1+m_2} \log\Big( 1+ \frac{n}{m_1+m_2}\Big)
\leq C \sum_{j=0}^{2n} \log\Big( 1+ \frac{n}{j}\Big)\\
&\leq C \int_{x=1}^{2n} \log\Big( 1+ \frac{n}{x}\Big) \rd x
\leq n \int_{1/k}^n \log(1+y) \frac{\rd y}{y^2} \leq Cn.
\end{align*}
Therefore $D_{n,3} =O( n)$ and thus 
the total contribution of the terms with $v\geq 3$ 
is $O(n^{2}\log(n)^{4\alpha -4})$.

%
\smallskip
\noindent\emph{Case $v=2$.}
Letting $M_n(4\alpha)$ be defined as usual, we have for some $q$ that
$l_1, \dots l_{2\alpha} \in [k_q, k_{q+1}]$. Denoting by $J_n(w, \alpha)$ the contribution of a single term with $v=2$
\begin{align}
J_n(w,\alpha)
&=
\sum_{M_n(4\alpha)} 
\prod_{\substack{1\leq k\leq 4\alpha -1\\k\neq q, q+ 2\alpha}} \rP[S(m_k)= A_k w]
\notag\\
&\qquad \times 
\Big[ \rP\Big(S(m_q)+S(m_{q+2\alpha})=  K_1 w\Big) - \rP\Big(S(m_q)+\cdots +S(m_{q+2\alpha})= K_2 w\Big)\Big]
\label{eq:SRWvar2},
\end{align}
where $K_1, K_2$ are integers determined by $\mathbf{k}, \mathbf{l}$ and their interlacement.
By \eqref{lclt_exact} it follows that
\begin{align*}
J_n(w,\alpha)
&\leq Cn \log(n)^{2\alpha -2} \sum_{p_0, \dots, p_{2\alpha}}
	\frac{1}{p_2 \cdots p_{2\alpha}} \Big[ \frac{1}{p_0+p_1}- \frac{1}{p_0+p_1+ \cdots +p_{2\alpha}} \Big]\\
&=Cn \log(n)^{2\alpha -2} \sum_{p_0, \dots, p_{2\alpha}}
	\frac{p_2+ \cdots + p_{2\alpha}}{p_2 \cdots p_{2\alpha} (p_0+p_1) (p_0+p_1+ \cdots + p_{2\alpha})} \\
&\leq C\alpha n \log(n)^{2\alpha -2} \sum_{p_0, \dots, p_{2\alpha}}
\frac{1}{p_3 \cdots p_{2\alpha} (p_0+p_1)(p_0+p_1+p_2+ \cdots + p_{2\alpha})} \\
&\leq C\alpha n \log(n)^{2\alpha -2} \sum_{p_2, \dots, p_{2\alpha}}\sum_{j=0}^{2n}
\frac{1}{p_3 \cdots p_{2\alpha} (j+p_2+ \cdots + p_{2\alpha})} \\
&\leq C \alpha n \log(n)^{2\alpha -2} \log(n)^{2\alpha -3+1} \sum_{p_1, p_2=0}^n\frac{1}{p_1+p_2}
\leq C \alpha n^2 \log(n)^{4\alpha -4}.
\end{align*}
Thus the total contribution of terms with interlacement index $v =2$  is $O\Big( n^2 \log(n)^{4\alpha -4}\Big)$.

To complete the proof of Proposition~\ref{lem:lln}, we first use Chebyshev's inequality to prove convergence along subsequences $n=\lfloor \rho^k\rfloor$, for $0<\rho<1$. We can fill in the gaps following the standard trick, as in 
\cite{cerny07}.
\end{proof}

\section{Proof of Theorem \ref{thm: relative complexity of RWRS}} \label{sec: proof of main theorem}

Our proof follows closely the outline of the proof of \cite{Aaronson} and \cite{Katok Thouvenot}. The main difference in our approach is that we are using the a.s. F\"olner property of the range and that we substitute the role of the local times with Theorem \ref{thm:distribution}. In the following we assume that the entropy of $\cS$ is finite. The case of infinite entropy can be easily derived by the same method.

Fix a finite generator $\beta$ for $\cS$, the existence of which is
a consequence of Krieger's Finite Generator Theorem \cite{Krieger} for $d=1$ and \cite{Katsnelson Weiss, Danilenko Park} for $d=2$. Let $\alpha=\left\{ \left[x_{1}\right]:x\in\Omega\right\} $
be the partition of $\Omega$ according to the first coordinate. The
partition $\Upsilon:=\alpha\times\beta$ is a countable generating
partition of $\Omega\times Y$ for $T$. Thus by Aaronson's Generator Theorem (Theorem \ref{thm: Aaronson's RC}), what we need to show is that 
\[
\frac{\log n}{n}\log\mathcal{K}_{\mathcal{B}_{\Omega}\times \mathcal{Y}}\left(\Upsilon,n,\epsilon\right)\xrightarrow{m} \pi h\left(\cS\right)\cdot\begin{cases}
\gamma, & d=1,{\bf \ A1}\\
2\sqrt{\det\Sigma}, & d=2,\ {\bf A2}, {\bf A2'}.
\end{cases}
\]
For $a_0,a_1,..,a_n\in \Upsilon$ we write 
\[
\left[a_0,a_1,\cdots,a_n\right]:=\cap_{j=0}^n T^{-j} a_j
\]
and the $\bar{d}_n$ metric on $\bigvee_{j=0}^{n-1} T^{-j}\Upsilon$, 
\[
\bar{d}_n \left( \left[a_0,a_1,\cdots,a_{n-1}\right],\left[a'_0,a'_1,\cdots,a_{n-1}\right] \right):=\frac{\# \left\{0\leq j\leq n-1: a_j\neq a'_j \right\}}{n}.
\]

It is straightforward to check that for all $n\in\mathbb{N}$ and $(w,y)\in\Omega\times Y$, 
\[
\left(\bigvee_{j=0}^{n-1}T^{-j}\Upsilon\right)(w,y)=\left[w_{0}^{n-1}\right]\times\beta_{R_{n}(w)}(y).
\]
where $\beta_{R_{n}(w)}(y):=\left(\bigvee_{l\in R_{n}(w)}\cS_{l}^{-1}\beta\right)(y)$
and $R_{n}(w):=\left\{ \sum_{j=1}^{l}w_{j}:\ 1\leq l\leq n\right\} $
is the range of the random walk up to time $n$. For $n\in\NN$, define
$\Pi_{n}:\Omega\to2^{\bigvee_{j=0}^{n-1}T^{-j}P}$ by 
\begin{eqnarray*}
\Pi_{n}(w): & = & \left\{ a\in\left(\bigvee_{j=0}^{n-1}T^{-j}\Upsilon\right):\ m\left(a\left|\BB_{\Omega}\times Y\right.\right)(w)>0 \right\}. 
\end{eqnarray*}
These are the partition elements seen by $w$. The function 
\begin{eqnarray*}
\Phi_{n,\epsilon}(x): & = & \min\left\{ \#F:\ F\subset\Pi_{n}(x),\ m\left(\cup_{a\in F}a\left|\BB_{\Omega}\times Y\right.\right)>1-\epsilon\right\} ,
\end{eqnarray*}
is an upper bound for $\mathcal{K}_{\mathcal{B}_{\Omega}\times \mathcal{Y}}\left(\Upsilon,n,\epsilon\right)(x)$
since in the definition of $\Phi_{n,\epsilon}$ we are using all sequences
in $\Pi_{n}(x)$ on their own and not grouping them into balls.

To get a lower bound, introduce 
\[
\mathcal{Q}_{n,\epsilon}(x):=\max\left\{ \#\left\{ z\in\Pi_{n}(x):\ \bar{d}_{n}(a,z)\leq\epsilon\right\} :\ a\in\Pi_{n}(x)\right\} 
\]
to be the maximal cardinality of elements of $\Pi_{n}(x)$ at a $\bar{d}_{n}$
ball centred at some $a\in\Pi_{n}(x)$. It then follows that 
\[
\mathcal{K}_{\BB_{\Omega}\times Y}\left(P,n,\epsilon\right)(x)\geq\frac{\Phi_{n,\epsilon}(x)}{\mathcal{Q}_{n,\epsilon}(x)}.
\]

Therefore the proof is separated into two parts. Firstly we prove that
\begin{equation}
\frac{\log n}{n}\log\Phi_{n,\epsilon}\xrightarrow{m}\pi h\left(\cS\right)\cdot\begin{cases}
\gamma, & d=1,{\bf \ A1}\\
2\sqrt{\det\Sigma}, & d=2,\ {\bf A2}, {\bf A2'},
\end{cases}\label{eq: upper bound}
\end{equation}
and the second part consists of showing that 
\begin{equation}
\frac{\log n}{n}\mathcal{\log Q}_{n,\epsilon}(x)\xrightarrow{m}0.\label{eq:lower bound}
\end{equation}

We will deduce \eqref{eq: upper bound} from the following Shannon Mcmillan Breiman Theorem. 

\begin{lemma} 
For $\mathbb{P}$ almost every $w\in\Omega$, 
\[
-\frac{\log n}{n}\log\nu\left(\beta_{R_{n}(w)}(y)\right)\xrightarrow{m}\pi h(\cS)\cdot\begin{cases}
\gamma, & d=1,{\bf \ A1}\\
2\sqrt{\det\Sigma}, & d=2,\ {\bf A2}, {\bf A2'},\ \text{as}\ n\to\infty.
\end{cases} \]
\label{lem: SMB for the Range}
\end{lemma}
\begin{proof}
Let $d\in{1,2}$. By Theorem \ref{thm:folner}, for $\PP$ almost every $w$, the range $\left\{ R_{n}(w)\right\} $
is a F\"{o}lner sequence for $\mathbb{Z}^d$. Whence by Kieffer's Shannon-McMillan-Breiman Theorem \cite{Kieffer}, for $\mathbb{P}$ a.e. $w$, 
\[
-\frac{1}{\# R_{n}(w)}\log\nu\left(\beta_{R_{n}(w)}(y)\right) \xrightarrow[n\to\infty]{\nu}h\left(\cS\right)
\]
and thus by Fubini, 
\[
-\frac{1}{\# R_{n}(w)}\log\nu\left(\beta_{R_{n}(w)}(y)\right)\xrightarrow[n\to\infty]{m}h\left(\cS\right).
\]
Notice that $h(\cS,\beta)=h(\cS)$ since $\beta$ is a generating partition. Since by \cite{Dvoretzky Erdos} and \eqref{Erdos Dvoretsky for Cauchy},
\[
\frac{\log n}{n}\# R_{n}(w)\xrightarrow[n\to\infty]{a.s.}\pi\begin{cases}
\gamma, & d=1,{\bf \ A1},\\
2\sqrt{\det\Sigma}, & d=2,\ {\bf A2}, {\bf A2'},
\end{cases}
\]
the conclusion of the lemma follows. \end{proof}
To keep the notations shorter, write 
\[
\mathtt{b} _d(n):= \frac{\pi n}{\log(n)}\begin{cases}
\gamma, & d=1,{\bf \ A1},\\
2\sqrt{\det\Sigma}, & d=2,\ {\bf A2}, {\bf A2'}.
\end{cases}
\]
\begin{proof}[Proof of \eqref{eq: upper bound}]
Let $\epsilon>0$  and for $n\in\mathbb{N}, x\in\Omega$ let
\[
H_{n,x,\epsilon}:=\left\{y\in Y:\ \nu\left(\beta_{R_n(x)} \right)(y)=e^{-\mathtt{b} _d(n)h(\cS)(1\pm\epsilon)} \right\}
\]
By Lemma \ref{lem: SMB for the Range}, there exists $N_\epsilon$ such that for all $n>N_\epsilon$, $\exists G_{n,\epsilon}\in\mathcal{B}_\Omega$ so that $\mathbb{P}\left( G_{n,\epsilon} \right)>1-\epsilon$ and for all $x\in G_{n,\epsilon}$,
\begin{equation}
\nu\left( H_{n,x,\epsilon}\right) >1-\frac{\epsilon}{2}. \label{eq: consequence of SMB}
\end{equation}
For $x\in G_{n,\epsilon}$, set $F_{n,x,\epsilon}:=\left\{\beta_{R_n(x)}(y):\ y\in H_{n,x,\epsilon} \right\}$. Since  
\[
\min\left\{\log\nu(a):a\in F_{n,x,\epsilon}\right\}>-\mathtt{b}_d(n)h(\cS)(1+\epsilon) 
\]
one has by a standard counting argument that for $x\in G_{n,\epsilon}$
\[
\log\Phi_{n,\epsilon} (x)\leq \log \# F_{n,x,\epsilon}\leq \mathtt{b}_d(n)h(\cS)(1+\epsilon)
\]

On the other hand, it follows from  \eqref{eq: consequence of SMB} that for small $\epsilon$  and $x\in G_{n,\epsilon}$, if $F\subset \Pi_n (x)$ with
$m \left( \left. \bigcup_{a\in F} a \right| \mathcal{B}_\Omega \times Y \right)(x)>1-\epsilon$ then  for large $n$
\[
\# F \geq \frac{1-3\epsilon/2}{\max\left\{\log\nu(a):a\in F_{n,x,\epsilon}\right\}}\geq \frac{e^{\mathtt{b}_d(n)h(\cS)(1-\epsilon)}}{2}
\]  
Thus for every $x\in G_{n,\epsilon}$ with $n$ large,
\[
\log\Phi_{n,\epsilon}(x)\geq \mathtt{b}_d (n)h(\cS)(1-\epsilon)+\log(1/2)\geq \mathtt{b}_d (n)h(\cS)(1-2\epsilon). 
\]
The conclusion follows since
\begin{equation*}
m\left( \left[ \log\Phi_{n,\epsilon}(x)= \mathtt{b}_d (n)h(\cS)(1\pm 2\epsilon) \right]\right) \geq \mathbb{P}\left( G_{n,\epsilon} \right)\xrightarrow[n\to\infty, \epsilon\to 0]{} 1. \qedhere
\end{equation*}
\end{proof} 

\subsection{Proof of Equation~\eqref{eq:lower bound}}
Let $\varepsilon>0$ and choose $\delta>0$ such that 
\[
2H(3\delta/2)+3\delta \log (\#\beta)<\varepsilon
\]
where for $0<p<1$,
\[
H(p)=-p\log_2 (p)-(1-p)\log_2 (1-p).
\]
is the entropy appearing in the Stirling approximation for the binomial coefficients.
It follows from Theorem \ref{thm:distribution} that there exists $\mathsf{c}>0$ and sets $A_{\delta,n}\in\BB_{\Omega}$ (for all large $n$)  such that for every $w\in A_{\delta,n}$, 
\[
\frac{\#\left\{ x\in R_{n}(w):\ l(n,x)(w)>\mathsf{c}\log(n)\right\} }{\#R_{n}(w)}>1-\delta,
\]
and  $\mathbb{P}\left(A_{\delta,n}\right)>1-\delta$. Since $\# R_n(w)\sim \mathtt{b}_d(n)$ almost surely we can assume further that for all $w\in A_{\delta,n}$, $\# R_n (w)\lesssim 2\mathtt{b}_d (n)$. 

Since $\Pi_n(w)\subset \left[ x_0^{n-1}\right] \times \beta_{R_n(w)}$, we can define a map $\mathbf{z}: \Pi_n(w)\to \beta_{R_n(w)}$ by 
\[
a=:\left[ x_0^{n-1}\right]\times \mathbf{z}(a)
\]
For $z\in \beta_{R_n(w)}$ and $j\in R_n(w)$, denote by $z_j$ the element of $\beta$ such that $z\subset \cS_{j}^{-1}\beta$. 

\begin{lemma}
For large $n\in\mathbb{N}$ and $w\in A_{\delta,n}$, if $a,a'\in\Pi_n(w)$ then
\[
\#\left\{ j\in R_n(w):\ \mathbf{z}(a)_j\neq \mathbf{z}\left(a'\right)_j\right\}\leq \mathtt{b}_d(n) \left(\frac{\bar{d}_n\left( a,a'\right)}{\hat{\mathsf{c}}}+2\delta\right), 
\]where $\hat{\mathsf{c}}:=\mathsf{c}\cdot \mathtt{b}_d(n)\log(n)/{n}$. 
\label{lem: first step to lower bound} 
\end{lemma}
\begin{proof}
Define 
\[
K_n(w):=\left\{ j\in R_n(w):\ \mathbf{z}(a)_j\neq \mathbf{z}\left(a'\right)_j\right\}
\] and
\[
F_n(w):=\left\{j\in R_n(w): l(n,x)(w)\geq \mathsf{c}\log(n)\right\}.
\] Then $K_n\subset \left(K_n\cap F_n\right) \cup F_n ^{c}$ and therefore since $w\in A_{\delta,n}$,
\begin{align*}
\# K_n(w) &\leq \# \left(K_n\cap F_n\right)(w)+\#F_n ^{c}(w)\\
&\leq \# \left(K_n\cap F_n\right)(w)+\delta \# R_n(w)\\
&\lesssim \# \left(K_n\cap F_n\right)(w)+2\delta \mathtt{b}_d (n)
\end{align*} 
 Finally,
 \begin{align*}
 \# \left(K_n\cap F_n\right) &\leq \frac{1}{\mathsf{c}\log(n)}\sum_{j\in F_n(w)} l(n,j)\mathbf{1}_{K_n(w)} \\
 &\leq \frac{1}{\mathsf{c}\log(n)} \# \left\{ 0\leq i\leq n-1:\ \mathbf{z}(a)_{s_i(w)}\neq \mathbf{z}(a')_{s_i(w)}\right\} \\
 &= \frac{n}{\mathsf{c}\log(n)} \bar{d}_n \left( a,a' \right).  
 \end{align*}
 The conclusion follows.
\end{proof}
\begin{proof}[Proof of \eqref{eq:lower bound}]
First we show that for $n$ large enough so that $A_{\delta,n}$ is defined,
\[
\max_{w \in A_{\delta,n}} \log \mathcal{Q}_{n,\hat{\mathsf{c}} \delta}(w) \leq \varepsilon a_d(n) 
\label{eq: what matters}
\]
To see this first notice that by Lemma \ref{lem: first step to lower bound} for every $a\in\Pi_n(w)$,
\[
\left\{ a'\in\Pi_n(w): \bar{d}_n\left(a,a'\right)\leq\hat{\mathsf{c}}\delta \right\} \subset \left\{ \mathbf{z}\in \beta_{R_n(w)}:\ \#\left\{j\in R_n(w): \mathbf{z}(a)_j\neq \mathbf{z}_j\right\} \leq 3\delta \mathtt{b}_d(n) \right\}.  
\]
Thus for $w\in A_{\delta,n}$, using the Stirling approximation for the Binomial and $\# R_n(w)\lesssim 2\mathtt{b}_d(n)$,
\begin{align*}
 \log \mathcal{Q}_{n,\hat{\mathsf{c}} \delta}(w) &\leq \log  \left[ \binom{\#R_n(w)}{3\delta \mathtt{b}_d(n)} (\# \beta)^{3\delta \mathtt{b}_d(n)} \right] \\
 &\lesssim 3\mathtt{b}_d(n) \delta \log (\#\beta)+\log\binom{2\mathtt{b}_d(n)}{3\delta \mathtt{b}_d(n)}\\
 &\sim \mathtt{b}_d(n)\left[ 3\delta \log (\#\beta)+ 2H(3\delta /2)\right]\\
 &\leq \varepsilon \mathtt{b}_d(n). 
 \end{align*}
 This shows that for large $n$,
 \[
 \mathbb{P}\left( \log \mathcal{Q}_{n,\hat{\mathsf{c}}\delta}>2\varepsilon \mathtt{b}_d(n) \right) \leq \mathbb{P} \left(A_{\delta,n}^c \right)\leq \delta
 \]
and thus we have finished the proof of \eqref{eq:lower bound}. 
\end{proof}
As was mentioned before, Theorem \ref{thm: relative complexity of RWRS} follows from \eqref{eq: upper bound} and \eqref{eq:lower bound}. 

\appendix
\section{Proofs of auxiliary results}
\begin{proof}[Proof of Lemma~\ref{lem:llt}]
We only prove the second statement, the first being simpler.
For $d=1$ and any $\epsilon >0$ by strong aperiodicity for $|t|>\epsilon$ it is true that $|\phi(t)| < C(\epsilon)<0$. Therefore
\begin{align*}
\Big|\rP[S(m)=0] - \rP[S_m=w]\Big|
&\leq \int_{-\pi}^\pi |1-\re^{\ri t w}| |\phi(t)|^m \rd t\\
&\leq \int_{|t|<\epsilon} |1-\re^{\ri t w}| |\phi(t)|^m \rd t 
	+4\pi C(\epsilon)^m,
\end{align*}
where the second term decays exponentially. 
For the first term we have, since 
$\phi(t) = 1-\gamma|t| + o(|t|)$, for $\epsilon$ small enough and $|t|< \epsilon$, 
$$
|\phi(t)|\leq |1-\gamma|t| | + D(\epsilon)|t| \leq 1-\frac{\gamma}{2} |t|.
$$
Therefore
\begin{align*}
\int_{|t|<\epsilon} |1-\re^{\ri t w}| |\phi(t)|^m \rd t 
&\leq C\int_{|t|<\epsilon} |t| |w| \Big(1-\frac{\gamma}{2}|t|\Big)^m \rd t \\
&=C |w|\int_{t=0}^\epsilon t \Big(1-\frac{\gamma t}{2}\Big)^m \rd t \\
&\leq C|w|\int_{t=0}^\epsilon t \exp\big(-\frac{m\gamma t}{2}\big) \rd t 
\leq C\frac{|w|}{m^2}.
\end{align*}
We prove \eqref{lclt_exact} for $d=1$. By \eqref{eq:potentialbound} it suffices to consider $w=0$. 
For the moment fix a small $\epsilon >0$.
Then, by aperiodicity, for $|t|>\epsilon$, there exists 
$\rho(\epsilon) \in (0,1)$, such that $|\phi(t)| < \rho(\epsilon)$.
Thus
\begin{align*}
\rP[S(n)=0]
&= \frac{1}{2\pi} \int_{-\pi}^{\pi} \phi(t)^n \rd t
= \frac{1}{2\pi} \int_{-\epsilon}^{\epsilon} \phi(t)^n \rd t + O( \rho(\epsilon)^n)\\
&= \frac{1}{2\pi} \int_{-\epsilon}^{\epsilon} [1-\gamma|t| +R(t)]^n \rd t + O( \rho(\epsilon)^n)\\
&=: I(n,\epsilon) + O( \rho(\epsilon)^n).
\end{align*}
Since $R(t) = o(t)$, for $|t|<\epsilon$ we can find $C(\epsilon)$ such that
$|R(t)|\leq C(\epsilon)|t|$ and such that $C(\epsilon) \to 0$ as $\epsilon \to 0$. 

Therefore letting $\gamma_1(\epsilon):= \gamma(1+C(\epsilon))$
\begin{align*}
I(n,\epsilon)
&\geq \frac{1}{2\pi} \int_{-\epsilon}^{\epsilon} [1-\gamma|t| - C(\epsilon) |t|]^n \rd t
= \frac{1}{\pi} \int_{0}^{\epsilon} [1-\gamma_1(\epsilon)t]^n \rd t\\
&=\frac{1}{\pi \gamma_1(\epsilon)} \int_{0}^{\gamma_1(\epsilon) \epsilon} [1-t]^n \rd t
=\frac{1}{\pi \gamma_1(\epsilon)} 
\Big\{ \frac{1}{n+1} - \frac{\big[1-\gamma_1(\epsilon)\epsilon\big]^{n+1}}{n+1}\Big\}.
\end{align*}
Since for $\epsilon >0$ small enough we have $0< 1-\gamma_1(\epsilon) \epsilon <1$
we compute 
$$
\liminf_{n\to \infty}
n \rP[S(n)=0] \geq \frac{1}{\pi \gamma_1(\epsilon)}.
$$
On the other hand we also have
\begin{align*}
I(n,\epsilon)
&\leq \frac{1}{2\pi} \int_{-\epsilon}^{\epsilon} [1-\gamma|t| + C(\epsilon) |t|]^n \rd t
= \frac{1}{\pi} \int_{0}^{\epsilon} [1-\gamma_2(\epsilon) t]^n \rd t
\end{align*}
where 
$\gamma_2(\epsilon) = 1-C(\epsilon)$.
Thus
\begin{align*}
I(n,\epsilon)
&\leq \frac{1}{\pi} \int_{0}^{\epsilon} [1-\gamma_2(\epsilon) t]^n \rd t
= \frac{1}{\pi\gamma_2(\epsilon)} \int_{0}^{\gamma_2(\epsilon)\epsilon} 
[1-t]^n \rd t\\
&= \frac{1}{\pi\gamma_2(\epsilon)} 
\Big\{ \frac{1}{n+1} - \frac{\big[1-\gamma_2(\epsilon)\epsilon\big]^{n+1}}{n+1}\Big\}.
\end{align*}
For $\epsilon>0$ small enough we have that 
$1-\gamma_2(\epsilon)\epsilon \in (0,1)$, and therefore we obtain that
$$\limsup_{n\to \infty} n \rP[S(n) = 0]
\leq \frac{1}{\pi \gamma_2(\epsilon)}.$$
Since $\epsilon>0$ is can be arbitrarily small and 
$\gamma_1(\epsilon), \lim \gamma_2(\epsilon) \to \gamma$, 
\eqref{lclt_exact} follows.

For $d=2$ the proof is similar, using polar coordinates.
\end{proof}
\medskip
\begin{proof}[Proof of Lemma~\ref{lem:lambdaasymptotic}]
Let $\delta >0$ be arbitrary but small. Then
\begin{align*}
\frac{1}{2\pi}\int_{t=-\pi}^\pi \frac{\lambda \phi (t) \rd t}{1-\lambda \phi(t)}
&= \frac{1}{2\pi}\int_{|t|\leq \delta} \frac{\lambda \phi (t) \rd t}{1-\lambda \phi(t)}+\frac{1}{2\pi}\int_{\pi\geq|t|> \delta} \frac{\lambda \phi (t) \rd t}{1-\lambda \phi(t)}.
\end{align*}
By strong aperiodicity for small enough $\delta>0$ there exists a small positive constant $D(\delta)$ such that 
$|\phi(t)|< 1- D(\delta)$ when $|t|>\delta$. Thus
$$\Big|\frac{1}{2\pi}\int_{\pi\geq |t|> \delta} \frac{\lambda \phi (t) \rd t}{1-\lambda \phi(t)}\Big| \leq C D(\delta)^{-1},$$
for all $\lambda \leq 1$. 
Also 
\begin{align*}
\frac{1}{2\pi}\int_{|t|\leq \delta} \frac{\lambda \phi (t) \rd t}{1-\lambda \phi(t)}
&= \frac{1}{2\pi}\int_{|t|\leq \delta} \frac{\lambda \phi (t) \rd t}{1-\lambda (1-\gamma|t|)} + 
I(\lambda, \delta),
\end{align*}
where a standard argument using \textbf{A1} and the strong aperiodicity shows that
there exists $r(\delta)=o_{\delta}(1)$, as $\delta \to 0$ such that 
$$|I(\lambda, \delta)| \leq r(\delta) \log \Big( \frac{1}{1-\lambda} \Big).$$

Finally as $\lambda \uparrow 1$ it is easily seen that
\begin{align*}
\frac{1}{2\pi}\int_{|t|\leq \delta} \frac{\lambda \phi (t) \rd t}{1-\lambda (1-\gamma|t|)}
&\sim \frac{1}{\pi}\int_{t=0}^\delta \frac{\rd t}{1-\lambda  +\lambda\gamma t }\\
&= \frac{1}{\pi}\int_{t=0}^\delta \frac{\rd t}{1-\lambda  +\lambda\gamma t }
\sim \frac{1}{\pi\gamma} \log\Big(\frac{1}{1-\lambda}\Big).
\end{align*}
Therefore as $\lambda \uparrow 1$
\begin{align*}
\frac{1}{2\pi}\int_{t=-\pi}^\pi \frac{\lambda \phi (t) \rd t}{1-\lambda \phi(t)}
&=\frac{1}{\pi\gamma} \log\Big(\frac{1}{1-\lambda}\Big) (1+ O(r(\delta))) + O (1)\\
&\sim \frac{1}{\pi\gamma} \log\Big(\frac{1}{1-\lambda}\Big),
\end{align*}
since $\delta$ is arbitrarily small and $r(\delta) \to 0$ as $\delta \to 0$.

\end{proof}

\setlength{\bibsep}{0.3pt}
\bibliographystyle{plainnat}

\end{document}